\documentclass[11pt,reqno]{amsart}

\newtheorem{proposition}{Proposition}[section]

\newtheorem{theorem}[proposition]{Theorem}

\theoremstyle{definition}

\newcommand{\thlabel}[1]{\label{th:#1}}
\newcommand{\thref}[1]{Theorem~\ref{th:#1}}
\newcommand{\selabel}[1]{\label{se:#1}}
\newcommand{\seref}[1]{Section~\ref{se:#1}}

\newcommand{\eqlabel}[1]{\label{eq:#1}}
\newcommand{\equref}[1]{(\ref{eq:#1})}

\def\ot{\otimes}

\def\NN{{\mathbb N}}

\def\ZZ{{\mathbb Z}}

\newcommand{\Cc}{\mathcal{C}}

\def\*C{{}^*\hspace*{-1pt}{\Cc}}
\def\text#1{{\rm {\rm #1}}}

\input xy
\xyoption {all} \CompileMatrices

\usepackage{amssymb}
\usepackage{color,amssymb,graphicx,amscd,txfonts}

\begin{document}

\title[Bicrossed products with the Taft algebra]
{Bicrossed products with the Taft algebra}

\author{A. L. Agore}
\address{''Simion Stoilow'' Institute of Mathematics of the Romanian Academy, P.O. Box 1-764, 014700 Bucharest, Romania \textbf{and} Vrije Universiteit Brussel, Pleinlaan 2, B-1050 Brussels, Belgium}
\email{ana.agore@vub.be and ana.agore@gmail.com}
\author{L. N\u ast\u asescu}
\address{'Simion Stoilow'' Institute of Mathematics of the Romanian Academy, P.O. Box 1-764, 014700 Bucharest, Romania}
\email{lauranastasescu@gmail.com}


\thanks{This work was  supported by a grant of Romanian Ministery of Research and Innovation, CNCS - UEFISCDI, project number PN-III-P1-1.1-TE-2016-0124, within PNCDI III. The first named author is a fellow of FWO (Fonds voor Wetenschappelijk Onderzoek -- Flanders). The authors gratefully acknowledge the hospitality of Max
Planck Institute f\"{u}r Mathematik (Bonn), where part of this work was done.}


%\maketitle

\begin{abstract}
Let $G$ be a group which admits a generating set consisting of finite order elements. We prove that any Hopf algebra which factorizes through the Taft algebra and the group Hopf algebra $K[G]$ (equivalently, any bicrossed product between the aforementioned Hopf algebras) is isomorphic to a smash product between the same two Hopf algebras. The classification of these smash products is shown to be strongly linked to the problem of describing the group automorphisms of $G$. As an application, we completely describe by generators and relations and classify all bicrossed products between the Taft algebra and the group Hopf algebra $K[D_{2n}]$, where $D_{2n}$ denotes the dihedral groups.
\end{abstract}

\maketitle

\section*{Introduction}

The bicrossed product construction first emerged in group theory,
under the name of Zappa-Sz\' ep product or knit product,  as a
natural generalization of the semi-direct product; see for
instance \cite[Section 2]{Takeuchi} where the terminology adopted
in the sequel originates. More precisely, any bicrossed product of
groups is constructed from a so-called matched pair of groups,
i.e. a pair of groups $(G,\,H)$ which act on each other by means
of two actions subject to some very natural compatibility
conditions (see \cite[Definition 2.1]{Takeuchi}). The two actions
of a matched pair induce a group structure on the direct product
$G \times H$ of the underlying sets, called the bicrossed product
of $G$ and $H$. Furthermore, it was shown that the bicrossed
product provides the answer to the factorization problem in the
sense that any group which factorizes through two given groups is
isomorphic to a bicrossed product between the same groups. Due to
the interest generated by its connection to the factorization
problem, the bicrossed product construction was considered in
varied contexts (e.g. for (co)algebras \cite{CIMZ}, \cite{cap},
Lie algebras \cite{Y}, \cite{majid3}, Hopf algebras \cite{majid}
etc.) and nowadays is being studied for many different purposes.
For instance, in the case of associative algebras this
construction is often referred to as twisted tensor product and it
seems to provide the proper algebraic counterpart for the product
of noncommutative spaces under the duality between the categories
of commutative algebras and algebraic affine spaces (see
\cite{ggv}, \cite{Jara} for further details). Another strong
motivation for studying bicrossed products stems from their link
to the (quantum) Yang-Baxter equation; more precisely, there are
several methods to produce solutions for the Yang-Baxter equation
by means of matched pairs and the corresponding bicrossed products
(see e.g. \cite{agv}, \cite{GIM}). Furthermore, the bicrossed
product construction turns out to be general enough to recover as
special cases many of the classical constructions in group theory
(e.g. semi-direct products), associative algebras (e.g. skew group
algebras, Ore extensions), Lie algebras (complex product
structures \cite{AS}) or Hopf algebras (e.g. the Drinfel'd double,
the generalized quantum double). However, even in the simplest
cases the problem of describing and classifying all bicrossed
products between two given objects (groups, associative algebras,
Lie algebras etc.) turns out to be very difficult. Probably the
most conclusing argument comes from the group case where the
description and classification of all bicrossed products between
two finite cyclic groups is still an open question (see
\cite{acim} for an up-to-date account on the subject).
Nevertheless, this topic has been intensively  studied in several
papers including \cite{a1}, \cite{abm}, \cite{Gabi}, \cite{CIMZ},
\cite{cap}, \cite{ggv}, \cite{Jara}, to mention only a few, and it
turned out to be an effective method of constructing new objects
(groups, associative algebras, Hopf algebras etc.) out of given
ones.

The present study continues the work of \cite{abm} where a
strategy for classifying bicrossed products of Hopf algebras was
proposed. This avenue of investigation based on the description of
all morphisms between two bicrossed products was previously
pursued in \cite{a1}, \cite{a2}, \cite{Gabi}, \cite{K} and led to
promising results. The paper is structured as follows.
\seref{prel} gives a brief review of some preparatory results used
in the sequel while the main results are contained in \seref{2}.
To start with, if $G$ is a group which admits a generating set
consisting of finite order elements, then \thref{main1} proves
that any bicrossed product between the Taft algebra and the group
Hopf algebra $K[G]$ is in fact a smash product between the same
two Hopf algebras. Although we are able to prove that the right
action of any such bicrossed product is always trivial, the
description of the left action depends essentially on the
relations, if any, satisfied by the generators of $G$. Along the
way, we describe by generators and relations all matched pairs
between the Taft algebra and the group Hopf algebra
$K[D_{ln}^{k}]$, where $D_{ln}^{k}$ is the semi-direct product of
cyclic groups defined in \equref{000.1}. More precisely, these are
implemented by two roots of unity $\beta \in  U_{l}(K) $, $\sigma
\in U_{(n, \,k-1)}(K)$ and the corresponding smash products,
denoted by $T_{lnm^{2}}^{\beta,\,\gamma}(q)$, are described by
generators and relations in \thref{main2}.

The classification of the smash products between the Taft algebra
and the group Hopf algebra $K[G]$ is shown to be strongly linked
to the problem of describing the group automorphisms of $G$
(\thref{izo2}). Since the problem of describing the automorphisms
of a given group is widely open and it leads to cumbersome
computations even for the groups $D_{ln}^{k}$ (see \cite{gg}), in
this paper we restrict to the classification of Hopf algebras
which factorize through the Taft algebra and the group Hopf
algebra $K[D_{2n}]$. More precisely, \thref{mainclasif} provides
necessary and sufficient conditions for two bicrossed products
between the Taft Hopf algebra and $K[D_{2n}]$ to be isomorphic
while \thref{mainclasif2} counts the number of types of these
bicrossed products which is shown to depend on the arithmetics of
$m$ and $n$.

It is worth mentioning that some of the Hopf algebras constructed in this paper as bicrossed products have shown up as part of other classification problems (see \cite{a1}, \cite{chen}, \cite{DIN}, \cite{iovanov} and \cite{LL}). However, the framework considered in the aforementioned papers does not allow for an explicit classification as the one performed in \thref{mainclasif} or \thref{mainclasif2}.

\section[Section 1]{Preliminaries}\selabel{prel}
Throughout $K$ will be a field. Unless specified otherwise, all algebras, coalgebras, bialgebras,
Hopf algebras, tensor products and homomorphisms are over $K$. For
a coalgebra $(C,\, \Delta,\, \varepsilon)$ and $y \in C$, we use Sweedler's $\Sigma$-notation: $\Delta(y) =
y_{(1)}\ot y_{(2)}$, $(I\ot\Delta)\Delta(y) = y_{(1)}\ot
y_{(2)}\ot y_{(3)}$, etc (summation understood). Let $A$ and $H$
be two Hopf algebras. $H$ is called a left (resp. right) $A$-module (co)algebra if $H$ is a (co)algebra in the monoidal category
of left (resp. right) $A$-modules.

For a positive integer
$n$ we denote by $U_n (K)$ the cyclic group of $n$-th roots of unity in $K$ and by $|U_n (K)|$ its order. Whenever $|U_n (K)|  = n$, any generator of $U_n (K)$ is called a primitive $n$-th root of unity.  $C_{n}$ stands for the cyclic group of order $n$ while $(m,\,n)$ denotes the greatest common divisor of $m$, $n \in \NN^{*}$. Furthermore, given integers $k$, $l$, $n \in \NN^{*}$ with $k^{l} \equiv 1 ({\rm mod}\, n)$ we denote by $D_{ln}^{k}$ the following semi-direct product of finite cyclic groups, often called split
metacyclic groups in the literature (\cite{gg}):
\begin{equation}\eqlabel{000.1}
C_{n} \rJoin_{k} C_{l} \,=\, < c,\, d\, |\, c^{l} = 1,\, d^{n} = 1,\, cd = d^{k}c >
\end{equation}
If $l = 2$ and $k = n-1$ we recover the dihedral groups which will be denoted simply by $D_{2n}$.
Given a group $G$, $K[G]$ stands for the group Hopf algebra on $G$. If $H$ is a Hopf algebra then $\mathcal{G}(H)$ denotes the group-like elements of $H$. Furthermore, for any $g$, $h \in \mathcal{G}(H)$ we denote by $\mathcal{P}_{g,\,h}(H)$ the set of $(g,\,h)$ - skew primitive elements of $H$, i.e. $y \in \mathcal{P}_{g,\,h}(H)$ if $ \Delta(y) = y \ot g + h \ot y$.

Consider $m \in \NN$, $m \geq 2$; since the Taft Hopf algebra $T_{m^{2}}(q)$ plays a central role in this paper, we will assume that the base field $K$ contains a primitive $m$-th root of unity $q$. Recall that the Taft Hopf algebra of order $m$ over $K$ is generated as an algebra by two elements $h$ and $x$
subject to the relations $h^{m} = 1$, $x^{m} = 0$ and $xh = qhx$. The coalgebra structure is such that $h$ is a group-like element and $x$ is a $(h,\,1)$ - skew primitive element.
Furthermore, $\{h^{i}x^{j}\}_{0 \leq i,\,j \leq m-1}$ is a $K$-linear basis of
the Taft algebra, the set of group-like elements is
$\mathcal{G}\bigl(T_{m^{2}}(q)\bigl) = \{h^{i} \, |\, i \in 0, \cdots,\, m-1 \}$ and the $(h^{j},\,
1)$-skew primitive elements are given as follows for any $j =
\{0,\,1, \cdots,\, m-1\}$:
\begin{equation}\eqlabel{01}
\mathcal{P}_{h^{j},\, 1} \bigl(T_{m^{2}}(q)\bigl) = \left \{\begin{array}{rcl}
\alpha(h^{j} - 1), \, & \mbox { if }& j \neq 1\\
\beta(h-1)+ \gamma x, \, & \mbox { if }& j=1
\end{array} \right.,\,\, {\rm for}\,\, {\rm some}\,\, \alpha, \, \beta,\, \gamma \in K.
\end{equation}

A \textit{matched pair} of Hopf algebras is a
quadruple $(A, H, \triangleleft, \triangleright)$, where $A$ and $H$
are Hopf algebras, $\triangleleft : H \otimes A \rightarrow H$,
$\triangleright: H \otimes A \rightarrow A$ are linear maps
such that $(A, \triangleright)$ is a left $H$-module coalgebra,
$(H, \triangleleft)$ is a right $A$-module coalgebra and the
following compatibilities hold for any $a$, $b\in A$, $y$, $z\in
H$.
\begin{eqnarray}
y \triangleright1_{A} &{=}& \varepsilon_{H}(h)1_{A}, \,\, 1_{H}
\triangleleft a =
\varepsilon_{A}(a)1_{H} \eqlabel{mp1} \\
y \triangleright(ab) &{=}& (y_{(1)} \triangleright a_{(1)}) \bigl
( (y_{(2)}\triangleleft a_{(2)})\triangleright b \bigl)
\eqlabel{mp2} \\
(y z) \triangleleft a &{=}& \bigl( y \triangleleft (z_{(1)}
\triangleright a_{(1)}) \bigl) (z_{(2)} \triangleleft a_{(2)})
\eqlabel{mp3} \\
y_{(1)} \triangleleft a_{(1)} \otimes y_{(2)} \triangleright
a_{(2)} &{=}& y_{(2)} \triangleleft a_{(2)} \otimes y_{(1)}
\triangleright a_{(1)} \eqlabel{mp4}
\end{eqnarray}
Any matched pair of Hopf algebras $(A, H, \triangleleft, \triangleright)$ gives rise to a new Hopf algebra on
the tensor coalgebra $A\ot H$ with the multiplication and antipode given for all $a$, $b \in A$, $y$, $z\in H$ as follows:
\begin{eqnarray*}
(a \bowtie y) \cdot (b \bowtie z) &=& a\, (y_{(1)}\triangleright
b_{(1)}) \bowtie (y_{(2)} \triangleleft b_{(2)})\, z\\
S_{A \bowtie H} ( a \bowtie y ) &=& S_H (y_{(2)}) \triangleright S_A
(a_{(2)}) \, \bowtie \, S_H (y_{(1)}) \triangleleft S_A (a_{(1)})
\end{eqnarray*}
where we denote $a\ot y$ by $a\bowtie y$. This new Hopf algebra is called the \textit{bicrossed product} associated to the matched pair of Hopf algebras $(A, H, \triangleleft, \triangleright)$ and it will be denoted by $A \bowtie H$. The left (resp. right) action $\triangleright: H \otimes A \rightarrow
A$ (resp. $\triangleleft : H
\otimes A \rightarrow H$) is called trivial if $y \triangleright a = \varepsilon_H(y) a$ (resp. $y \triangleleft a = \varepsilon_A
(a) y$) for all $a\in A$ and $y\in H$.

The smash product of Hopf algebras (also called semi-direct product) appears as a special case of the bicrossed product construction. More precisely, if $(A, \triangleright)$ is a left $H$-module coalgebra and we
consider $H$ as a right $A$-module coalgebra via the trivial
action, i.e. $y \triangleleft a = \varepsilon_A(a) y$, then $(A,
H, \triangleleft, \triangleright)$ is a matched pair of Hopf
algebras if and only if $(A, \triangleright)$ is also a left
$H$-module algebra and the following compatibility condition holds for all $y \in H$ and $a\in A$:
\begin{eqnarray*}
y_{(1)} \otimes y_{(2)} \triangleright a =  y_{(2)} \otimes
y_{(1)} \triangleright a.
\end{eqnarray*}
In this case, the associated
bicrossed product $A\bowtie H = A\# H$ is the left version of the
\emph{smash (semi-direct) product} of Hopf algebras. Thus, $A\# H$ is
the tensor coalgebra $A\ot H$, with the following multiplication:
\begin{eqnarray*}
(a \# y) \cdot (b \# z):= a \, (y_{(1)} \triangleright b )\,\#\,
y_{(2)}\, z
\end{eqnarray*}
for all $a$, $b\in A$, $y$, $z\in H$, where we denote $a\ot y$ by
$a\# y$. Symmetrically, we can define the right version of the
smash product of Hopf algebras by considering $A$ to be a left
$H$-module coalgebra via the trivial action.

Other important examples of bicrossed products include the classical Drinfel'd double (\cite[Theorem
IX.3.5]{Kassel}) as well as the generalized quantum double  (\cite[Example 7.2.6]{majid}). We refer the reader to \cite{majid} for more details and further examples.

Bicrossed products were introduced in connection to the factorization problem. More precisely, any Hopf algebra which factorizes through two given Hopf algebras is isomorphic to a bicrossed product between the same two Hopf algebras (\cite{majid3}). Recall that the factorization problem for Hopf algebras consists of describing and classifying all Hopf algebras which factorize through two given Hopf algebras, i.e. all Hopf algebras for which there exist injective Hopf algebra maps $i : A \to E $ and $j : H\to E$ such that the map $A \ot H \to E, \quad a \ot y \mapsto i(a) j(y)$ is bijective.

\section[Section 2]{Main results}\selabel{2}

In this section we investigate bicrossed products between the Taft algebra and the group Hopf algebra of certain groups or, equivalently, Hopf algebras which factorize through the aforementioned Hopf algebras. More precisely, our first result states that if the group $G$ admits a generating set consisting of finite order elements then any bicrossed product between the Taft algebra and $K[G]$ is in fact a smash product.

\begin{theorem}\thlabel{main1}
Let $G$ be a group generated by finite order elements. Then any Hopf algebra which factorizes through the Taft algebra $T_{m^{2}}(q)$ and the group Hopf algebra $K[G]$ is isomorphic to a smash product between the same two Hopf algebras.
\end{theorem}
\begin{proof}
Consider $S$ to be a generating set of $G$ consisting of finite order elements and let $g \in S$. According to our assumption we have ${\rm ord}(g) = u$, where $u \in \NN^{*}$.

Let $(T_{m^{2}}(q), K[G], \triangleleft, \triangleright)$ be a
matched pair of Hopf algebras where $\triangleleft : K[G] \ot T_{m^{2}}(q) \to K[G]$, $\triangleright : K[G] \ot T_{m^{2}}(q) \to T_{m^{2}}(q)$ are coalgebra maps satisfying the compatibility conditions \equref{mp1}-\equref{mp4}. We will prove that the right action $\triangleleft$ is trivial. Indeed, by \cite[Lemma 5.1]{abm} or directly by module coalgebra axioms, we have $g \triangleright h \in \mathcal{G}(T_{m^{2}}(q))$, i.e. $g \triangleright h = h^{i}$ for some $i \in \{0, 1, \cdots, m-1\}$. Assume first that $i = 0$; by induction one obtains $1 = g^{u} \triangleright h = h$ which is an obvious contradiction. Thus we have $i \in \{1, 2, \cdots, m-1\}$. Suppose now that $i \neq 1$;  in this case \cite[Lemma 5.1]{abm} gives $g \triangleright x \in \mathcal{P}_{g\, \triangleright h,\, g\, \triangleright 1} \bigl(T_{m^{2}}(q)\bigl) = \mathcal{P}_{h^{i},\, 1} \bigl(T_{m^{2}}(q)\bigl)$ and since $i \neq 1$ it follows from \equref{01} that $g \triangleright x = \alpha (h^{i} - 1)$ for some $\alpha \in K$. By induction we obtain $x = g^{u}  \triangleright x = \alpha g^{u-1}  \triangleright h^{i} - \alpha$ and we have reached a contradiction as \cite[Lemma 5.1]{abm} implies $g^{u-1}  \triangleright h^{i} \in \mathcal{G}\bigl(T_{m^{2}}(q)\bigl)$. Hence we are led to $g \triangleright h = h$. Furthermore, this implies $g \triangleright x = \alpha (1 - h) + \beta \, x$ for some $\alpha$, $\beta \in K$. Again by induction we arrive at $x = g^{u} \triangleright x = \alpha (1 + \beta + \dots + \beta^{u-1}) (1-h) + \beta^{u} x$ which yields:
\begin{eqnarray}
\alpha (1 + \beta + \dots + \beta^{u-1}) = 0,\quad\quad \beta^{u}  = 1. \eqlabel{111.2}
\end{eqnarray}
For the rest of the proof we will focus on the right action $\triangleleft$. Using \cite[Lemma 5.1]{abm} we obtain $g \triangleleft h \in \mathcal{G}(K[G])$. Remark that the case $g \triangleleft h = 1$ can be easily ruled out by the same arguments used for the left action. Hence $g \triangleleft h = g_{i_{1}}^{a_{1}} g_{i_{2}}^{a_{2}} \cdots g_{i_{s}}^{a_{s}}$ for some $s$, $a_{1}, \cdots, a_{s} \in \NN^{*}$ and $g_{i_{1}}, \cdots, g_{i_{s}} \in S$. Moreover, again by \cite[Lemma 5.1]{abm} we have $g \triangleleft x \in \mathcal{P}_{g \triangleleft h,\, g \triangleleft 1}(K[G]) = \mathcal{P}_{g_{i_{1}}^{a_{1}} g_{i_{2}}^{a_{2}} \cdots g_{i_{s}}^{a_{s}},\, g}(K[G])$, i.e. $g \triangleleft x = \gamma \Bigl(g - g_{i_{1}}^{a_{1}} g_{i_{2}}^{a_{2}} \cdots g_{i_{s}}^{a_{s}}\Bigl)$ for some $\gamma \in K$. Now we apply \equref{mp4} for the pair $(g,\,x)$ and we obtain:
$$g \triangleleft  x \ot h + g \ot g \triangleright x = g \triangleleft  h \ot g \triangleright x + g \triangleleft  x \ot 1.$$
Furthermore, this gives:
$$
\gamma g \ot h - \gamma g_{i_{1}}^{a_{1}} g_{i_{2}}^{a_{2}} \cdots g_{i_{s}}^{a_{s}} \ot h + \alpha g \ot 1 - \alpha g \ot h + \beta g \ot x =
$$
\begin{equation}\eqlabel{111.3}
\alpha g_{i_{1}}^{a_{1}} g_{i_{2}}^{a_{2}} \cdots g_{i_{s}}^{a_{s}} \ot 1 - \alpha g_{i_{1}}^{a_{1}} g_{i_{2}}^{a_{2}} \cdots g_{i_{s}}^{a_{s}} \ot h + \beta g_{i_{1}}^{a_{1}} g_{i_{2}}^{a_{2}} \cdots g_{i_{s}}^{a_{s}} \ot x + \gamma g \ot 1 - \gamma g_{i_{1}}^{a_{1}} g_{i_{2}}^{a_{2}} \cdots g_{i_{s}}^{a_{s}} \ot 1
\end{equation}
In particular, \equref{111.3} implies $\beta g \ot x = \beta g_{i_{1}}^{a_{1}} g_{i_{2}}^{a_{2}} \cdots g_{i_{s}}^{a_{s}} \ot x$ and since by \equref{111.2} we have  $\beta \neq 0$ it follows that $g_{i_{1}}^{a_{1}} g_{i_{2}}^{a_{2}} \cdots g_{i_{s}}^{a_{s}} = g$. Therefore, we have $g \triangleleft h = g$ and $g \triangleleft x = 0$.

Next, by using the compatibility condition \equref{mp2} we arrive at the following:
\begin{eqnarray*}
g \triangleright hx &\stackrel{\equref{mp2}} {=}& (g \triangleright h) \Bigl((g \triangleleft h) \triangleright x\Bigl) = h (g \triangleright x) = \alpha (h - h^{2}) + \beta hx\\
g \triangleright xh &\stackrel{\equref{mp2}} {=}& (g \triangleright x_{(1)}) \Bigl((g \triangleleft x_{(2)}) \triangleright h\Bigl) = (g \triangleright x) \Bigl((g \triangleleft h) \triangleright h\Bigl) + (g \triangleright 1) \Bigl((g \triangleleft x) \triangleright h\Bigl)\\
&=&  \alpha (h - h^{2}) + \beta xh = \alpha (h - h^{2}) + \beta q hx
\end{eqnarray*}
As $xh = q hx$, using the above computation gives:
$$
q\alpha h - q\alpha h^{2} + q\beta hx = \alpha h - \alpha h^{2} + \beta xh
$$
and since $q \neq 1$ we obtain $\alpha = 0$. Thus $g \triangleright x = \beta x$.

Note that the above arguments apply to any finite order element $g_{i}$ of $G$ and we denote by $\beta_{i}$ the corresponding $u_{i}$ root of unity, where $u_{i} = {\rm ord}(g_{i})$. We are now ready to prove that the right action $\triangleleft : K[G] \ot T_{m^{2}}(q) \to K[G]$ is indeed trivial. To this end, it can easily be seen, using induction and \equref{mp3}, that  for all $\nu \in \{0,\, 1,\, \cdots,\, u_{i}-1\}$ we have $g_{i}^{\nu}  \triangleleft h = g_{i}^{\nu}$. Hence, we obtain:
\begin{eqnarray}\eqlabel{11.11}
g_{i}^{\nu}  \triangleleft h^{j} = g_{i}^{\nu},\,\, {\rm for}\, {\rm all}\, \nu \in \{0,\, 1,\, \cdots,\, u_{i}-1\}\, {\rm and}\, j  \in \{0,\, 1,\, \cdots,\, m-1\}.
\end{eqnarray}
Similarly, using \equref{mp2}, yields:
\begin{eqnarray}\eqlabel{22.22}
g_{i}^{\nu}  \triangleright h^{j} = h^{j},\,\, {\rm for}\, {\rm all}\, \nu \in \{0,\, 1,\, \cdots,\, u_{i}-1\}\, {\rm and}\, j  \in \{0,\, 1,\, \cdots,\, m-1\}.
\end{eqnarray}
By putting together \equref{11.11} and  \equref{22.22} gives:
\begin{eqnarray*}
g_{i_{1}}^{\nu_{1}} g_{i_{2}}^{\nu_{2}}  \triangleleft h^{j} \stackrel{\equref{mp3}} {=} \bigl(g_{i_{1}}^{\nu_{1}} \triangleleft (g_{i_{2}}^{\nu_{2}}
\triangleright h^{j}) \bigl) (g_{i_{2}}^{\nu_{2}} \triangleleft h^{j}) \stackrel{\equref{22.22}} {=}  (g_{i_{1}}^{\nu_{1}} \triangleleft h^{j}) (g_{i_{2}}^{\nu_{2}} \triangleleft h^{j}) \stackrel{\equref{11.11}} {=} g_{i_{1}}^{\nu_{1}} g_{i_{2}}^{\nu_{2}}
\end{eqnarray*}
for all $g_{i_{1}}$, $g_{i_{2}} \in S$,  $\nu_{1} \in \{0,\, 1,\, \cdots,\, u_{i_{1}}-1\}$, $\nu_{2} \in \{0,\, 1,\, \cdots,\, u_{i_{2}}-1\}$ and $j  \in \{0,\, 1,\, \cdots,\, m-1\}$.
Inductively, we obtain:
\begin{eqnarray}\eqlabel{33.33}
g_{i_{1}}^{\nu_{1}} g_{i_{2}}^{\nu_{2}} \cdots g_{i_{t}}^{\nu_{t}}  \triangleleft  h^{j}  = g_{i_{1}}^{\nu_{1}} g_{i_{2}}^{\nu_{2}} \cdots g_{i_{t}}^{\nu_{t}}
\end{eqnarray}
for all $g_{i_{l}} \in S$, $\nu_{l} \in \{0,\, 1,\, \cdots,\, u_{i_{l}}-1\}$, $l \in \{1,\, 2,\, \cdots,\, t \}$ and $j  \in \{0,\, 1,\, \cdots,\, m-1\}$.
Furthermore, we also have:
$$
g_{i}^{2}  \triangleleft x \stackrel{\equref{mp3}} {=} \bigl(g_{i} \triangleleft (g_{i}
\triangleright x_{(1)}) \bigl) (g_{i} \triangleleft x_{(2)}) = \bigl(g_{i} \triangleleft (g_{i}
\triangleright x) \bigl) (g_{i} \triangleleft h) + \bigl(g_{i} \triangleleft (g_{i}
\triangleright 1) \bigl) (g_{i} \triangleleft x) =  0
$$
and using induction again we get $g_{i}^{\nu}  \triangleleft x^{j} = 0$ for all $\nu \in \{0,\, 1,\, \cdots,\, u_{i}-1\}$ and $j  \in \{1,\,2,\, \cdots,\, m-1\}$. Now remark that we actually proved that:
\begin{eqnarray}
g_{i}^{\nu}  \triangleleft x^{j} = g_{i}^{\nu} \varepsilon(x^{j}),\,\, {\rm for}\, {\rm all}\, \nu \in \{0,\, 1,\, \cdots,\, u_{i}-1\}\, {\rm and} \, j  \in \{0,\, 1,\, \cdots,\, m-1\}.
\end{eqnarray}
As before, one can easily prove that:
\begin{eqnarray}\eqlabel{44.44}
g_{i_{1}}^{\nu_{1}} g_{i_{2}}^{\nu_{2}} \cdots g_{i_{t}}^{\nu_{t}}  \triangleleft  x^{j}  = g_{i_{1}}^{\nu_{1}} g_{i_{2}}^{\nu_{2}} \cdots g_{i_{t}}^{\nu_{t}} \varepsilon(x^{j})
\end{eqnarray}
for all $g_{i_{l}} \in S$, $\nu_{l} \in \{0,\, 1,\, \cdots,\, u_{i_{l}}-1\}$, $l \in \{1,\, 2,\, \cdots,\, t \}$ and $j  \in \{0,\, 1,\, \cdots,\, m-1\}$.
Finally, by putting all the above together yields:
\begin{eqnarray*}
g_{i_{1}}^{\nu_{1}} g_{i_{2}}^{\nu_{2}} \cdots g_{i_{t}}^{\nu_{t}}  \triangleleft  h^{j}  x^{j'} = \bigl(g_{i_{1}}^{\nu_{1}} g_{i_{2}}^{\nu_{2}} \cdots g_{i_{t}}^{\nu_{t}}  \triangleleft  h^{j} \bigl)\, \triangleleft\,  x^{j'}
\stackrel{\equref{33.33}, \equref{44.44}} {=} g_{i_{1}}^{v_{1}} g_{i_{2}}^{v_{2}} \cdots g_{i_{t}}^{v_{t}}  \, \varepsilon(x^{j'})
\, =\, g_{i_{1}}^{v_{1}} g_{i_{2}}^{v_{2}} \cdots g_{i_{t}}^{v_{t}}  \, \varepsilon(h^{j}  x^{j'})
\end{eqnarray*}
for all $g_{i_{l}} \in S$, $\nu_{l} \in \{0,\, 1,\, \cdots,\, u_{i_{l}}-1\}$, $l \in \{1,\, 2,\, \cdots,\, t \}$, $j$, $j'  \in \{0,\, 1,\, \cdots,\, m-1\}$, i.e. the right action is indeed trivial and the proof is now finished.
\end{proof}

Notice that although \thref{main1} shows that the right action of
any bicrossed product between the Taft algebra $T_{m^{2}}(q)$ and
the group Hopf algebra $K[G]$ is trivial whenever $G$ admits a
generating set consisting of finite order elements, the
description of the left action is only narrow as it depends
essentially on the relations, if any, satisfied by the generators
of the group $G$. In order to illustrate this idea, in the sequel
we will focus on the case where $G = D_{ln}^{k}$ as defined in
\equref{000.1}. We will be able to describe completely by
generators and relations all Hopf algebras which factorize through
the Taft algebra and the group Hopf algebra $K[D_{ln}^{k}]$. We
start by computing all matched pairs $(T_{m^{2}}(q),
K\bigl[D_{ln}^{k}\bigl], \triangleleft, \triangleright)$: as a
consequence of \thref{main1} the right action is trivial while the
left action is implemented by two roots of unity $\beta \in
U_{l}(K)$, $\sigma \in U_{(n, \,k-1)}(K)$.

\begin{theorem} \thlabel{main2}
There exists a bijection between the set of all
matched pairs $(T_{m^{2}}(q), K\bigl[D_{ln}^{k}\bigl], \triangleleft,
\triangleright)$ and $U_{l}(K) \times U_{(n, \,k-1)}(K)$ such that the matched pair
$(\triangleleft, \triangleright)$ corresponding to $(\beta,\, \sigma) \in U_{l}(K) \times U_{(n, \,k-1)}(K)$ is given by:
\begin{equation}\eqlabel{mp}
d^{t} c^{i} \triangleright h^{j'} x^{j} = \beta^{ij}\,\sigma^{tj} h^{j'}\, x^{j}, \quad \quad \quad d^{t} c^{i} \triangleleft h^{j'}  x^{j}  = d^{t} c^{i}  \, \varepsilon(h^{j'}  x^{j})
\end{equation}
for all $t = 0, \cdots,\, n-1$, $i = 0, \cdots,\, l-1$ and $j$, $j' = 0, \cdots,\, m-1$.\\
Furthermore, any bicrossed product between $T_{m^{2}}(q)$ and $K\bigl[D_{ln}^{k}\bigl]$
is isomorphic to $T_{lnm^{2}}^
{\beta,\,\sigma}(q)$, for some $(\beta,\, \sigma) \in U_{l}(K) \times U_{(n, \,k-1)}(K)$, where $T_{lnm^{2}}^
{\beta,\,\sigma}(q)$ is the Hopf
algebra generated by $c$, $d$, $h$ and $x$ subject to the following relations:
\begin{eqnarray*}
&&\hspace*{2mm}c^{l} = d^{n} = h^{m} = 1, \quad x^{m} = 0,\\
&&c h = h  c, \quad d h =
h  d, \quad c d =
d^{k}  c,\\
&& \hspace*{-2mm} c x = \beta x c, \quad  d x = \sigma  x d,  \quad x h = q h x
\end{eqnarray*}
with the coalgebra structure and antipode given by:
$$
\Delta(c) = c \ot c, \quad \Delta(d) = d \ot d, \quad \Delta(h) = h \ot h, \quad \Delta(x) =
 x \ot h + 1 \ot x, \quad \varepsilon(x) = 0,
$$
$$
\varepsilon(a) = \varepsilon(b) = \varepsilon(h) = 1,\quad S(h) = h^{m-1},\quad S(x) = -xh^{m-1},\quad S(c) = c^{l-1},\quad S(d) = d^{n-1}.
$$
\end{theorem}

\begin{proof}
Let $(T_{m^{2}}(q), K\bigl[D_{ln}^{k}\bigl], \triangleleft, \triangleright)$ be a
matched pair. Since the group $D_{ln}^{k}$ is generated by two finite order elements, namely $c$ and $d$, we can conclude by (the proof of) \thref{main1} that the right action is trivial and the left action is given as follows for all $i \in \{0, 1, \cdots, l-1\}$ and $j \in \{0, 1, \cdots, n-1\}$:
\begin{eqnarray}
c^{i} \triangleright h = h, \,\, c^{i} \triangleright x = \beta^{i} x,\,\, d^{j} \triangleright h = h, \,\, d^{j} \triangleright x = \sigma^{j} x \,\,  {\rm where}\,\, \beta \in U_{l}(K), \,\, \sigma \in U_{n}(K).\eqlabel{action}
\end{eqnarray}
Furthermore, since the relation  $cdc^{l-1} = d^{k}$ holds in $D_{ln}^{k}$ we have $cdc^{l-1} \triangleright  x = d^{k} \triangleright x$ which after a straightforward computation gives $\sigma x = \sigma^{k} x$, i.e. $\tau \in U_{k-1}(K)$. Therefore we obtain $\sigma \in U_{(n, \,k-1)}(K)$ as desired. Finally, the formulae in \equref{mp} can be easily derived using the defining compatibility conditions of a matched pair of Hopf algebras (i.e. \equref{mp1} - \equref{mp4}) exactly as in the proof of  \thref{main1}.

Thus any bicrossed product between $T_{m^{2}}(q)$ and $K\bigl[D_{ln}^{k}\bigl]$ is in fact a smash product $T_{m^{2}}(q)\, \# \, K\bigl[D_{ln}^{k}\bigl]$ associated to the left action determined by a pair $(\beta,\, \sigma) \in U_{l}(K) \times U_{(n, \,k-1)}(K)$ as in \equref{action}.
Up to canonical identification, $T_{m^{2}}(q) \, \#\, K[C_n]$ is generated as an algebra by
$h = h\, \# \,1$, $x = x\, \# \, 1$, $c = 1\, \#\,
c$ and $d = 1\, \#
\, d$. For instance, we have:
\begin{eqnarray*}
&& ch = (1\, \# \, c) (h\, \# \,1)= (c \triangleright h)\, \# \,(c \triangleleft h) = h\, \# \,c =  (h\, \#\, 1)(1\, \# \, c) = hc\\
&& c x = (1\, \# \,c) (x\, \# \,1) = (a \triangleright x_{(1)})\,
\#\, (c \triangleleft x_{(2)}) = (c \triangleright x) \,\# \,(c \triangleleft h) + (c \triangleright 1) \,\# \,(c \triangleleft x) =\beta \, x \,\# \,c = \beta \, x c
\end{eqnarray*}
Since the rest of the defining relations of $T_{lnm^{2}}^
{\beta,\,\sigma}(q)$ can be checked in much the same fashion as above they are left to the reader.
\end{proof}

Our next goal is to classify the Hopf algebras described in \thref{main1}. We will see that actually the quest for classifying all bicrossed products between the Taft algebra and the group Hopf algebra $K[G]$, where $G$ is a group generated by a set of finite order elements, comes down to an old and notoriously difficult problem in group theory, namely that of describing the automorphisms of a certain given group. More precisely, our next result shows that any Hopf algebra isomorphism between two smash products as in \thref{main1} is in some sense induced by a group automorphism of $G$.

\begin{theorem}\thlabel{izo2}
Let $G$ be a group generated by a set of finite order elements and
consider two smash products $T_{m^{2}}(q)\, \#\, K[G]$ and
$T_{m^{2}}(q)\, \# '\, K[G]$. If $\varphi: T_{m^{2}}(q)\, \#\,
K[G] \to T_{m^{2}}(q)\, \# '\, K[G]$ is a Hopf algebra isomorphism
then there exist two unitary coalgebra maps $u:T_{m^{2}}(q) \to
T_{m^{2}}(q)$, $r: K[G] \rightarrow T_{m^{2}}(q)$ and a Hopf
algebra automorphism $v: K[G] \rightarrow K[G]$ such that $\varphi
= \varphi_{(u,\, r,\, v)}$, where $\varphi_{(u,\, r,\, v)}:
T_{m^{2}}(q)\, \#\, K[G] \to T_{m^{2}}(q)\, \# '\, K[G]$ is given
as follows for all $a \in T_{m^{2}}(q)$ and $t \in K[G]$:
$$
\varphi_{(u,\, r,\, v)}(a\, \#\, t) = u(a) r(t_{(1)})\, \# ' \, v(t_{(2)}).
$$
\end{theorem}
\begin{proof}
Consider $S = \{g_{i} \in G ~|~ i \in I \}$ to be a generating set
of $G$ with ${\rm ord}(g_{i}) = u_{i}$, where $u_{i} \in \NN^{*}$
for all $i \in I$. According to the proof of \thref{main1}, the left actions
$\triangleright : K[G] \ot  T_{m^{2}}(q) \to T_{m^{2}}(q)$, $
\triangleright ' : K[G] \ot  T_{m^{2}}(q) \to T_{m^{2}}(q)$,
corresponding to the two smash products $T_{m^{2}}(q)\, \#\, K[G]$
and $T_{m^{2}}(q)\, \# '\, K[G]$, respectively, are given as
follows for all $i \in I$:
\begin{eqnarray*}
g_{i} \triangleright h = h,\,\,\, g_{i} \triangleright x = \alpha_{i} \, x, \,\,\, g_{i} \triangleright ' h = h,\,\,\, g_{i} \triangleright ' x = \overline{\alpha}_{i} \, x,\,\,\, {\rm where} \,\,\, \alpha_{i},  \overline{\alpha}_{i}  \in U_{u_{i}}(K).
\end{eqnarray*}

By \cite[Theorem 3.2]{abm}, the set of all Hopf algebra morphisms between $T_{m^{2}}(q)\, \#\, K[G]$ and $T_{m^{2}}(q)\, \# '\, K[G]$ is in bijective correspondence with the set of all quadruples $(u, p, r, v)$, consisting of two
unitary coalgebra maps $u:T_{m^{2}}(q) \to T_{m^{2}}(q)$ and $r: K[G] \rightarrow T_{m^{2}}(q)$, and
two Hopf algebra maps $p: T_{m^{2}}(q) \to K[G]$ and $v: K[G] \rightarrow K[G]$
subject to the following compatibilities for all $a$, $b \in T_{m^{2}}(q)$, $t$, $w \in K[G]$:
\begin{eqnarray}
u(a_{(1)}) \ot p(a_{(2)}) &{=}& u(a_{(2)}) \ot p(a_{(1)})\eqlabel{C1}\\
u(ab) &{=}& u(a_{(1)}) \, \bigl( p (a_{(2)}) \triangleright' u(b) \bigl)\eqlabel{C3}\\
r(tw) &{=}& r(t_{(1)}) \, \bigl(v(t_{(2)}) \triangleright'
r(w)\bigl)\eqlabel{C4}\\
r(t_{(1)}) \bigl(v(t_{(2)}) \triangleright' u(a) \bigl) &{=}& u
(t_{(1)} \triangleright a_{(1)}) \, \Bigl( p (t_{(2)}
\triangleright a_{(2)}) \triangleright' r(t_{(3)})
\Bigl) \eqlabel{C5}\\
v(t) p (a) &{=}& p (t_{(1)} \triangleright a) v (t_{(2)}
\eqlabel{C6}
\end{eqnarray}
The correspondence is such that the morphism $\psi_{(u, p, r, v)}$ associated to $(u, p, r, v)$ is given as follows for all $a \in T_{m^{2}}(q)$ and $t \in K[G]$:
\begin{equation}\eqlabel{morfbicros}
\psi_{(u, p, r, v)}(a\, \# \,t) = u(a_{(1)}) \, \bigl( p(a_{(2)})
\triangleright' r(t_{(1)}) \bigl) \,\, \#' \,
p(a_{(3)}) v(t_{(2)}).
\end{equation}
We start by proving that the Hopf algebra morphism $p : T_{m^{2}}(q) \to K[G]$ of any quadruple $(u, p, r, v)$ satisfying the compatibilities \equref{C1} - \equref{C6} is trivial, i.e. $p(a) = \varepsilon(a) 1$ for all $a \in T_{m^{2}}(q)$. Indeed, since $p(h) \in \mathcal{G}\bigl(K[G]\bigl)$ we have $p(h) = g_{i_{1}}^{a_{1}} g_{i_{2}}^{a_{2}} \cdots g_{i_{s}}^{a_{s}}$ for some $s$, $a_{1}, \cdots, a_{s} \in \NN$ and $g_{i_{1}}, \cdots, g_{i_{s}} \in S$. Furthermore, this implies that $p(x) \in \mathcal{P}_{p(h),\,1}\bigl(K[G]\bigl)$, i.e. $p(x) = \beta \bigl(g_{i_{1}}^{a_{1}} g_{i_{2}}^{a_{2}} \cdots g_{i_{s}}^{a_{s}} - 1\bigl)$ for some $\beta \in K$. Hence, we have:
\begin{eqnarray*}
p(hx) = p(h)p(x) =\beta g_{i_{1}}^{a_{1}} g_{i_{2}}^{a_{2}} \cdots g_{i_{s}}^{a_{s}} \bigl(g_{i_{1}}^{a_{1}} g_{i_{2}}^{a_{2}} \cdots g_{i_{s}}^{a_{s}} - 1\bigl)= \beta  \bigl(g_{i_{1}}^{a_{1}} g_{i_{2}}^{a_{2}} \cdots g_{i_{s}}^{a_{s}} - 1\bigl)g_{i_{1}}^{a_{1}} g_{i_{2}}^{a_{2}} \cdots g_{i_{s}}^{a_{s}} = p(xh).
\end{eqnarray*}
As $p(xh) = q p(hx)$ we obtain $\beta (q-1) g_{i_{1}}^{a_{1}}
g_{i_{2}}^{a_{2}} \cdots g_{i_{s}}^{a_{s}} \bigl(g_{i_{1}}^{a_{1}}
g_{i_{2}}^{a_{2}} \cdots g_{i_{s}}^{a_{s}} - 1\bigl) = 0$ and
since $q$ is a primitive $m$-th root of unity we have $q \neq 1$
and all the above comes down to $\beta \bigl(g_{i_{1}}^{a_{1}}
g_{i_{2}}^{a_{2}} \cdots g_{i_{s}}^{a_{s}} - 1\bigl) = 0$, i.e.
$p(x) = 0$.

Next we investigate the unitary coalgebra maps $u:T_{m^{2}}(q) \to T_{m^{2}}(q)$. Firstly, we obviously have $u(1) =
1$ and $u(h)=h^{d}$, where $d \in \{\, 0,\, 1, \cdots, m-1 \, \}$. Now the compatibility condition \equref{C3} gives:
$$
u(h^{2}) = u(h) \bigl(p(h) \triangleright ' u(h)\bigl) = h^{d}
\bigl(g_{i_{1}}^{a_{1}} g_{i_{2}}^{a_{2}} \cdots g_{i_{s}}^{a_{s}} \triangleright ' h^{d}\bigl) = h^{2d}.
$$
By repeatedly applying  \equref{C3} one easily obtains $u(h^{j}) =
h^{jd}$ for any $j \in \NN$. Furthermore, by using \cite[Lemma 5.1]{abm}, we have $u(x) \in
\mathcal{P}_{h^{d},\, 1} \bigl(T_{m^{2}}(q)\bigl)$ which leads to $u(x) = \mu(h^{d}-1)$ for some $\mu \in K$ if $d \neq 1$ or to $u(x) = \mu(h-1)+ \gamma x$ for some $\mu$, $\gamma \in K$ in the case when $d=1$.

Assume first that $d \neq 1$; thus $u(x) = \mu(h^{d}-1)$ and the compatibility condition \equref{C3} yields:
\begin{eqnarray*}
u(hx) &\stackrel{\equref{C3}} {=}& u(h) \bigl(p(h) \triangleright '
u(x)\bigl) = h^{d} \bigl(g_{i_{1}}^{a_{1}} g_{i_{2}}^{a_{2}} \cdots g_{i_{s}}^{a_{s}} \triangleright '
\mu(h^{d}-1)\bigl) = \mu h^{d} (h^{d}-1)\\
u(xh) &\stackrel{\equref{C3}} {=}& u(x_{(1)}) \bigl(p(x_{(2)})
\triangleright ' u(h)\bigl) = u(x) \bigl(p(h) \triangleright ' u(h)\bigl) + p(x)
\triangleright ' u(h) = \mu h^{d} (h^{d}-1).
\end{eqnarray*}
Now since $u(xh) = q u(hx)$ the above computations amount to $\mu (q-1) h^{d}(h^{d}-1) = 0$. Moreover, as $q \neq 1$
we obtain $\mu(h^{d}-1) = 0$ and thus $u(x) = 0$. However, this contradicts the fact that the
corresponding Hopf algebra map $\psi_{(u,\,p,\,r,\,v)}$ was assumed to be an isomorphism. Indeed we have:
\begin{eqnarray*}
\psi_{(u,\,p,\,r,\,v)}(x \# 1) = u(x_{(1)}) \bigl(p(x_{(2)}) \triangleright '
r(1)\bigl) \# p(x_{(3)}) v(1)= u(x_{(1)}) \# p(x_{(2)}) = u(x) \# p(h) + 1 \# p(x) = 0.
\end{eqnarray*}
Therefore the only possibility left is $d = 1$ and $u(x)
= \mu (h-1) + \gamma x$. In this case, we obtain:
\begin{eqnarray*}
u(hx) \stackrel{\equref{C3}} {=} u(h) \bigl(p(h) \triangleright '
u(x)\bigl) = h \Bigl[g_{i_{1}}^{a_{1}} g_{i_{2}}^{a_{2}} \cdots g_{i_{s}}^{a_{s}} \triangleright ' \Bigl(\mu(h-1)+
\gamma x \Bigl)\Bigl] = \mu h(h-1)+ \gamma \alpha_{i_{1}}^{a_{1}} \alpha_{i_{2}}^{a_{2}} \cdots \alpha_{i_{s}}^{a_{s}} hx.
\end{eqnarray*}
A similar computation gives $u(xh) = \mu h(h-1) + \gamma q hx$ and
since $u(xh) = q u(hx)$ we obtain $\mu = q \mu$ and $q \gamma
\alpha_{i_{1}}^{a_{1}} \alpha_{i_{2}}^{a_{2}} \cdots
\alpha_{i_{s}}^{a_{s}} = q \gamma$. As $q \neq 1$ we must have
$\mu = 0$. Thus $\gamma \neq 0$ (otherwise we would have $u(x) =
0$ which leads to the same contradiction as in the case where $d
\neq 1$) and so $\alpha_{i_{1}}^{a_{1}} \alpha_{i_{2}}^{a_{2}}
\cdots \alpha_{i_{s}}^{a_{s}}  = 1$. Consequently, $u$ takes the
following form:
\begin{eqnarray*}
u(h^{i}) = h^{i}, \,\,\, u(x^{i}) = \gamma^{i} x^{i},\,\,\, \gamma \in K^{*}, \,\,\, i \in \{0, 1, \cdots, m-1\}.
\end{eqnarray*}

We are now in a position to prove that the Hopf algebra map $p:
T_{m^{2}}(q) \to K[G]$ is indeed trivial. To this end, by setting $a = x$ the compatibility condition \equref{C1}
reduces to the following:
\begin{eqnarray*}
&& u(x_{(1)}) \ot p(x_{(2)}) = u(x_{(2)}) \ot p(x_{(1)})\\
&\Leftrightarrow& u(x) \ot p(h) + u(1) \ot p(x) = u(h) \ot p(x) +
u(x) \ot p(1)\\
&\Leftrightarrow& \gamma x \ot g_{i_{1}}^{a_{1}} g_{i_{2}}^{a_{2}} \cdots g_{i_{s}}^{a_{s}}  = \gamma x \ot 1.
\end{eqnarray*}
Now recall that $\gamma \neq 0$ which leads to $g_{i_{1}}^{a_{1}} g_{i_{2}}^{a_{2}} \cdots g_{i_{s}}^{a_{s}}  = 1$ and thus $p$ is the
trivial map, i.e. $p(a) = \varepsilon(a) 1$ for all $a \in
T_{m^{2}}(q)$. Hence, the Hopf algebra isomorphism corresponding to the quadruple $(u, \varepsilon , r, v)$ as in \equref{morfbicros} takes the following simplified form for all $a \# t \in  T_{m^{2}}(q)\, \# \, K[G]$:
$$
\psi_{(u,\,\varepsilon,\,r,\,v)}(a \# t) =  u(a) r(t_{(1)})\, \# ' \, v(t_{(2)}).
$$
We are left to prove that $v$ is a Hopf algebra automorphism. Indeed, consider $\varphi : T_{m^{2}}(q) \bowtie ' K[G] \to T_{m^{2}}(q) \bowtie K[G]$ to be the inverse of $\psi_{(u,\,\varepsilon,\,r,\,v)}$. Using the same arguments as above we can conclude that there exists a quadruple $(\overline{u},\,\varepsilon,\, \overline{r},\, \overline{v})$ as in  \equref{morfbicros} such that $\varphi = \psi_{(\overline{u},\,\varepsilon,\, \overline{r},\, \overline{v})}$, i.e.:
$$
\psi_{(\overline{u},\,\varepsilon,\, \overline{r},\, \overline{v})}(a\, \# ' t) =  \overline{u}(a) \overline{r}(t_{(1)})\, \#  \, \overline{v}(t_{(2)}),\,\,\, {\rm for} \,\, {\rm all}\,\, a \# ' t \in  T_{m^{2}}(q)\, \# ' \, K[G].
$$
Then, for any $t \in K[G]$ we have:
\begin{eqnarray*}
1 \, \# ' \, t = \psi_{(u,\,\varepsilon,\,r,\,v)} \circ \psi_{(\overline{u},\,\varepsilon,\, \overline{r},\, \overline{v})}(1 \, \#  ' \, t) = \psi_{(u,\,\varepsilon,\,r,\,v)} \bigl(\overline{r}(t_{(1)})\, \#  \, \overline{v}(t_{(2)})\bigl) \, = u\bigl(\overline{r}(t_{(1)}) \bigl) r\bigl(\overline{v}(t_{(2)})\bigl)\, \# ' \, v\bigl(\overline{v}(t_{(3)})\bigl).
\end{eqnarray*}
By applying $\varepsilon \ot {\rm Id}$ to the above identity yields $v \circ \overline{v} (t) = t$ for any $t \in K[G]$. Similarly one can prove that $\overline{v} \circ v(t) = t$ and the proof is now finished.
\end{proof}

Obviously, the Hopf algebra automorphism $v :K[G] \to K[G]$ from \thref{izo2} is uniquely determined by a group automorphism of $G$. In the sequel we will provide a complete classification for the Hopf algebras which factorize through the Taft algebra and the group Hopf algebra $K[D_{2n}]$. To this end, recall that for any $n > 2$ we have ${\rm Aut} (D_{2n}) = \{\gamma_{s,t} ~|~ s,\, t = 0,\,1, \cdots,\, n-1,\, (t,\, n) = 1\}$, where $\gamma_{s,t} : D_{2n} \to D_{2n}$ is the group automorphism defined as follows:
\begin{eqnarray*}
\gamma_{s,t}(c) = d^{s} c,\,\,\,\,\, \gamma_{s,t}(d) = d^{t}.
\end{eqnarray*}

Consider $\beta$, $\overline{\beta} \in U_{2}(K)$, $\sigma$,
$\overline{\sigma} \in U_{(n,\,n-2)}(K)$ and let
$T_{2nm^{2}}^{\beta,\, \sigma}(q)$ and
$T_{2nm^{2}}^{\overline{\beta},\, \overline{\sigma}}(q)$,
respectively, to be the corresponding Hopf algebras as in
\thref{main2}. Our next result provides necessary and sufficient
conditions for the two aforementioned Hopf algebras to be
isomorphic.

\begin{theorem}\thlabel{mainclasif}
Let $m$, $n \in \NN^{*}$, $m  \geq 2$ and $n  \geq 3$. Then the Hopf algebras  $T_{2nm^{2}}^{\beta,\, \sigma}(q)$ and $T_{2nm^{2}}^{\overline{\beta},\, \overline{\sigma}}(q)$ are isomorphic if and only if there exist integers $f$, $F \in  \{0,\,1,\, \cdots, \, m-1\}$ and $s$, $t \in \{0,\,1,\, \cdots, \, n-1\}$ such that $(t,\,n) = 1$ and the following compatibilities hold:
\begin{eqnarray}
m ~|~ 2f, \quad m ~|~ nF, \quad m ~|~ 2F,\eqlabel{izo01}\\
q^{f} = \beta\, \overline{\beta}\, \overline{\sigma}^{s},\quad q^{F} = \sigma\,  \overline{\sigma}^{t}\eqlabel{izo02}.
\end{eqnarray}
\end{theorem}
\begin{proof}
In light of (the proof of) \thref{izo2} the Hopf algebras $T_{2nm^{2}}^{\beta,\, \sigma}(q)$ and $T_{2nm^{2}}^{\overline{\beta},\, \overline{\sigma}}(q)$ are isomorphic if and only if there exists a quadruple $(u,\, p,\, r,\, v)$ satisfying the compatibilities \equref{C1}-\equref{C6} such that the corresponding morphism $\psi_{(u, p, r, v)}: T_{2nm^{2}}^{\beta,\, \sigma}(q) \to T_{2nm^{2}}^{\overline{\beta},\, \overline{\sigma}}(q)$ given by \equref{morfbicros}  is an isomorphism.
It follows from \thref{izo2} that $\psi_{(u, p, r, v)}$ being an isomorphism implies $p = \varepsilon$ and $v \in {\rm Aut}_{{\rm Hopf}}(K[D_{2n}])$. Furthermore, as $p$ is the trivial map, the compatibility conditions \equref{C1} and \equref{C6} are trivially fulfilled while \equref{C3} comes down to  $u: T_{m^{2}}(q) \to T_{m^{2}}(q)$ being a Hopf algebra morphism. More precisely, exactly as in (the proof of) \thref{izo2}, we can conclude that $u$ is given as follows for all $i$, $j \in \{0, 1, \cdots, m-1\}$:
\begin{eqnarray*}
u(h^{i}x^{j}) = \gamma^{j} h^{i}x^{j},\,\,\, {\rm where} \,\, \gamma \in K^{*}.
\end{eqnarray*}
We are left to discuss the compatibility conditions  \equref{C4}
and \equref{C5}. To start with, the Hopf algebra automorphism $v :
K[D_{2n}] \to K[D_{2n}]$ and the coalgebra map $r: K[D_{2n}] \to
T_{m^{2}}(q)$ are uniquely determined by some integers $s$, $t \in
\{0,\,1,\, \cdots, \, n-1\}$, $f$, $F \in  \{0,\,1,\, \cdots, \,
m-1\}$, such that $(t,\,n) = 1$ and:
\begin{eqnarray*}
v(c) = d^{s} c,\quad v(d) = d^{t} \quad {\rm and} \quad
r(c) = h^{f}, \quad r(d) = h^{F}.
\end{eqnarray*}
Using the compatibility condition \equref{C4} yields:
\begin{eqnarray*}
r(c^{2}) \stackrel{\equref{C4}} {=} r(c) \bigl(v(c) \triangleright ' r(c)\bigl) \,= h^{f} \bigl(d^{s}c \triangleright ' h^{f} \bigl) \,= h^{2f}
\end{eqnarray*}
which implies $m ~|~ 2f$. Similarly, using induction and \equref{C4} we obtain $r(d^{i}) = h^{iF}$ for all $i \in \NN$. In particular, we get $1 = r(d^{n}) = h^{nF}$ and thus $m ~|~ nF$. Furthermore, we have:
\begin{eqnarray*}
r(cd) &\stackrel{\equref{C4}} {=}& r(c) \bigl(v(c) \triangleright ' r(d)\bigl)\, = h^{f} \bigl(d^{s}c \triangleright ' h^{F} \bigl)\, = h^{f+F}\\
r(d^{n-1} c) &\stackrel{\equref{C4}} {=}&  r(d^{n-1}) \bigl(v(d^{n-1}) \triangleright ' r(c)\bigl) \, = h^{(n-1)F} \Bigl((d^{t})^{n-1} \triangleright ' h^{f} \Bigl)\, = h^{(n-1)F+f}\,= h^{-F+f}.
\end{eqnarray*}
Therefore we obtain $h^{f+F} = h^{-F+f}$ which leads to $m ~|~ 2F$ as desired. Next we look at the compatibility condition \equref{C5}; as $p$ is the trivial map it comes down to the following:
\begin{eqnarray}\eqlabel{C5'}
r(t_{(1)}) \bigl(v(t_{(2)}) \triangleright' u(a) \bigl) \, = u
(t_{(1)} \triangleright a) r(t_{(2)}).
\end{eqnarray}
By considering $t = c$ and $a = x$ in \equref{C5'} yields:
\begin{eqnarray*}
&& r(c) \bigl(v(c) \triangleright ' u(x)\bigl) \, =  u(c \triangleright x)r(c)\\
&\Leftrightarrow& h^{f}  \bigl(d^{s}c \triangleright ' \gamma x \bigl) \, =  \gamma\, \beta\,  x h^{f}\\
&\Leftrightarrow& \gamma\, \overline{\beta}\, \overline{\sigma}^{s}\, h^{f} x  \, =  \gamma\, \beta\, q^{f} \, h^{f} x.
\end{eqnarray*}
Hence we have $q^{f} = \beta\, \overline{\beta}\, \overline{\sigma}^{s}$. Similarly, by setting $t = d$ and $a = x$ in \equref{C5'} gives $q^{F} = \sigma\,  \overline{\sigma}^{t}$.

Putting all together, we proved that if $T_{2nm^{2}}^{\beta,\, \sigma}(q)$ and $T_{2nm^{2}}^{\overline{\beta},\, \overline{\sigma}}(q)$ are isomorphic Hopf algebras then there exist integers $s$, $t \in  \{0,\,1,\, \cdots, \, n-1\}$ and $f$, $F \in  \{0,\,1,\, \cdots, \, m-1\}$ with $(t,\,n) = 1$ such that \equref{izo01} and \equref{izo02} are fulfilled.

Conversely, if there exist $s$, $t \in  \{0,\,1,\, \cdots, \, n-1\}$ and $f$, $F \in  \{0,\,1,\, \cdots, \, m-1\}$ satisfying the conditions above, we construct a coalgebra map $r_{(f,\,F)}: K[D_{2n}] \to T_{m^{2}}(q)$ and a Hopf algebra automorphism $v_{(s,\,t)} : K[D_{2n}] \to K[D_{2n}]$ as follows:
\begin{eqnarray}
r_{(f,\,F)}(d^{i}c^{j}) = h^{iF+jf}, \quad v_{(s,\,t)}(c) = d^{s} c,\quad v_{(s,\,t)}(d) = d^{t} \quad {\rm for} \quad {\rm all} \quad i,\,j \in \NN.
\end{eqnarray}
Then it can be readily proved that the following map is a Hopf algebra isomorphism:
\begin{eqnarray*}
\varphi := \psi_{({\rm Id},\, \varepsilon,\, r_{(f,\,F)},\, v_{(s,\,t)})}: T_{2nm^{2}}^{\beta,\, \sigma}(q) \to T_{2nm^{2}}^{\overline{\beta},\, \overline{\sigma}}(q), \quad\varphi(a\, \# \, y) = a\, r_{(f,\,F)}(y_{(1)})\, \#\, v_{(s,\,t)}(y_{(2)}).
\end{eqnarray*}
where ${\rm Id} : T_{m^{2}}(q) \to T_{m^{2}}(q)$ is the identity
map on $T_{m^{2}}(q)$. First notice that the first part of the
proof implies that $\varphi$ is a Hopf algebra map. We are left to
prove that $\varphi$ is in fact an isomorphism. To start with,
since $(t,\,n) = 1$ there exist $\tau$, $\tau ' \in \ZZ$ such that
$t \tau + n \tau ' = 1$. Moreover, we can find unique integers
$\alpha$, $\beta$, $\eta$, $\mu \in \ZZ$, $\overline{f}$,
$\overline{F} \in \{0,\,1,\, \cdots, \, m-1\}$ and $\tau_{1}$,
$\tau_{2} \in  \{0,\,1,\, \cdots, \, n-1\}$ such that:
\begin{eqnarray}
\tau = \eta n + \tau_{1}, \quad -\tau_{1} s = \mu n + \tau_{2},\eqlabel{**}\\
-f - F \tau_{2} = \alpha m + \overline{f}, \quad -F \tau_{1} = \beta m + \overline{F}\eqlabel{*&}.
\end{eqnarray}
We are now in a position to construct the inverse of $\varphi$. To this end, let $r_{(\overline{f},\,\overline{F})}: K[D_{2n}] \to T_{m^{2}}(q)$ be a coalgebra map and $v_{(\tau_{2},\,\tau_{1})} : K[D_{2n}] \to K[D_{2n}]$ a Hopf algebra map, defined as follows:
 \begin{eqnarray}
r_{(\overline{f},\,\overline{F})}(d^{i} c^{j}) = h^{i\overline{F} + j\overline{f}}, \quad v_{(\tau_{2},\,\tau_{1})}(c) = d^{\tau_{2}} c,\quad v_{(\tau_{2},\,\tau_{1})}(d) = d^{\tau_{1}} \quad {\rm for} \quad {\rm all} \quad i,\, j \in \NN.
\end{eqnarray}
The proof will be finished once we prove that the map $\overline{\varphi} := \psi_{({\rm Id},\, \varepsilon,\, r_{\overline{f},\,\overline{F}},\, v_{(\tau_{2},\,\tau_{1})})}$ is the inverse of $\varphi$. Indeed, for instance we have:
\begin{eqnarray*}
\varphi \circ \overline{\varphi} \,(h^{i}x^{j} \, \# \, d^{k}) &=& \varphi\, \Bigl(h^{i}x^{j}r_{\overline{f},\,\overline{F}}(d^{k}) \,\, \# \,\, v_{(\tau_{2},\,\tau_{1})}(d^{k}) \Bigl)\,=\, \varphi\, \Bigl(h^{i}x^{j} h^{k\overline{F}} \,\, \# \,\, d^{k \tau_{1}} \Bigl)\\
&=&h^{i}x^{j} h^{k\overline{F}} \, r_{(f,\,F)}\bigl(d^{k \tau_{1}} \bigl)\,\, \# \,\, v_{(s,\,t)}\bigl(d^{k \tau_{1}}\bigl)\,=\,h^{i}x^{j} h^{k\overline{F} + k \tau_{1} F} \,\, \# \,\, d^{k \underline{\tau_{1} t}}\\
&\stackrel{\equref{**}} {=}& h^{i}x^{j} h^{k\overline{F} + k \tau_{1} F} \,\, \# \,\, d^{k (\tau t - \eta n t)}\,=\, h^{i}x^{j} h^{k\overline{F} + k \tau_{1} F} \,\, \# \,\, d^{k \underline{\tau t}}\,=\, h^{i}x^{j} h^{k(\underline{\overline{F} + \tau_{1} F})} \,\, \# \,\, d^{k (1 - n \tau ')}\\
&\stackrel{\equref{*&}} {=}& h^{i}x^{j} h^{-k \beta m} \,\, \# \,\, d^{k} \,=\,  h^{i}x^{j} \,\, \# \,\, d^{k}
\end{eqnarray*}
for all $i$, $j \in \{0,\,1,\, \cdots, \, m-1\}$ and $k \in
\{0,\,1,\, \cdots, \, n-1\}$. By a rather long but straightforward
computation one can see that in fact $\varphi \circ
\overline{\varphi} = {\rm Id}_{T_{2nm^{2}}^{\overline{\beta},\,
\overline{\sigma}}(q)}$ and $\overline{\varphi} \circ \varphi =
{\rm Id}_{T_{2nm^{2}}^{\beta,\, \sigma}(q)}$.
\end{proof}

Based on \thref{mainclasif}, we are now able to count the number of types of Hopf algebras which factorize through $T_{m^{2}}(q) $ and $ K[D_{2n}]$.
\begin{theorem}\thlabel{mainclasif2}
Let $m$, $n \in \NN^{*}$, $m  \geq 2$, $n  \geq 3$ and consider $\sharp_{n,\,m}^{\,\,\,q}$ to be the number of
types of Hopf algebras which factorize through $T_{m^{2}}(q) $ and $ K[D_{2n}]$. Then, we have:
\begin{eqnarray*}
\sharp_{n,\,m}^{\,\,\,q} = \left\{\begin{array}{rcl} 2, &
\mbox{if}& m \,\, \mbox{and}\,\, n  \,\, \mbox{are}  \,\, \mbox{both} \,\, \mbox{odd} \\ 3, &
\mbox{if}& m \,\, \mbox{is}\,\, \mbox{odd}\,\, \mbox{and}\,\, n  \,\, \mbox{is}\,\, \mbox{even} \\  1, &
\mbox{if}& m \,\, \mbox{and}\,\, n  \,\, \mbox{are}  \,\, \mbox{both} \,\, \mbox{even} \,\, \mbox{or} \,\,  m \,\, \mbox{is}\,\, \mbox{even}\,\, \mbox{and}\,\, n  \,\, \mbox{is}\,\, \mbox{odd.}\\ \end{array}\right.
\end{eqnarray*}
\end{theorem}
\begin{proof}
According to \thref{main2}, any Hopf algebra which factorizes through $T_{m^{2}}(q) $ and $ K[D_{2n}]$ is isomorphic to $T_{2nm^{2}}^{\beta,\, \sigma}(q)$ for some $\beta \in U_{2}(K)$ and $\sigma \in U_{(n,\,n-2)}(K)$.

Suppose first that $n$ is odd. Then $\sigma \in U_{(n,\,n-2)}(K)$ implies $\sigma = 1$ and we are left to decide if the Hopf algebras  $T_{2nm^{2}}^{1,\, 1}(q)$ and $T_{2nm^{2}}^{-1,\, 1}(q)$ are isomorphic.
By \thref{mainclasif} we have an isomorphism between the aforementioned Hopf algebras if and only if there exist integers $f$, $F \in  \{0,\,1,\, \cdots, \, m-1\}$, $s$, $t \in \{0,\,1,\, \cdots, \, n-1\}$ such that $(t,\,n) = 1$ and the following compatibilities hold:
\begin{eqnarray}\eqlabel{nrizo}
m ~|~ 2f, \quad m ~|~ nF, \quad m ~|~ 2F,\quad
q^{f} = -1,\quad q^{F} = 1.
\end{eqnarray}
As $q$ is a primitive $m$-th root of unity and $q^{F} = 1$, where
$F \in  \{0,\,1,\, \cdots, \, m-1\}$, we obtain $F = 0$. If $m$ is
odd as well and $m ~|~ 2f$ it follows that $m ~|~ f$. Now since $f
\in  \{0,\,1,\, \cdots, \, m-1\}$ we get $f = 0$ which contradicts
$q^{f} = -1$. Consequently, if $n$ and $m$ are both odd then
$T_{2nm^{2}}^{1,\, 1}(q)$ is not isomorphic to $T_{2nm^{2}}^{-1,\,
1}(q)$. Hence and we have two types of Hopf algebras which
factorize through $T_{m^{2}}(q) $ and $ K[D_{2n}]$.\\ If $m$ is
even and $m = 2 m'$ for some $m' \in \NN^{*}$, it is
straightforward to see that the integers $s = t = 1$, $f = m'$ and
$F=0$ fulfill the conditions in \equref{nrizo}. Therefore, in this
case $T_{2nm^{2}}^{1,\, 1}(q)$ is isomorphic to
$T_{2nm^{2}}^{-1,\, 1}(q)$ and we obtain $\sharp_{n,\,m}^{\,\,\,q}
= 1$.

Assume now that $n$ is even and $n = 2 \overline{n}$ for some $ \overline{n} \in \NN^{*}$. Then $\beta$, $\sigma \in \{1, \, -1\}$ and the Hopf algebras which factorize through $T_{m^{2}}(q) $ and $ K[D_{2n}]$ are the following:
$$T_{2nm^{2}}^{1,\, 1}(q),\quad T_{2nm^{2}}^{-1,\, -1}(q),\quad T_{2nm^{2}}^{1,\, -1}(q), \quad T_{2nm^{2}}^{-1,\, 1}(q).$$
If $m$ is odd, $m = 2 \overline{m} + 1$ for some $\overline{m} \in \NN^{*}$, we will see that $T_{2nm^{2}}^{-1,\, -1}(q)$ is isomorphic to $T_{2nm^{2}}^{1,\, -1}(q)$ while any two of the remaining three Hopf algebras are non-isomorphic. Indeed, it can be readily seen that the integers $s = t = 1$, $f = F = 0$ fulfill the compatibility conditions \equref{izo01} and \equref{izo02} in \thref{mainclasif} which leads us to the conclusion that $T_{2nm^{2}}^{-1,\, -1}(q)$ is isomorphic to $T_{2nm^{2}}^{1,\, -1}(q)$. Similarly, by \thref{mainclasif}, the Hopf algebras $T_{2nm^{2}}^{1,\, 1}(q)$ and $T_{2nm^{2}}^{-1,\, -1}(q)$ are isomorphic if and only if there exist integers $f$, $F \in  \{0,\,1,\, \cdots, \, m-1\}$, $s$, $t \in \{0,\,1,\, \cdots, \, n-1\}$ such that $(t,\,n) = 1$ and the following hold:
\begin{eqnarray}\eqlabel{nrizo2}
m ~|~ 2f, \quad m ~|~ nF, \quad m ~|~ 2F,\quad
q^{f} = 1,\quad q^{F} = (-1)^{t}.
\end{eqnarray}
To start with, notice that $n$ even and $(t,\,n) = 1$ implies $t$ odd and therefore $ q^{F} = -1$. Moreover, as $m ~|~ 2F$ and $m = 2 \overline{m} + 1$ we obtain $m ~|~ F$. Now since $F \in  \{0,\,1,\, \cdots, \, m-1\}$ we get $F = 0$. However, this contradicts $ q^{F} = -1$ and therefore $T_{2nm^{2}}^{1,\, 1}(q)$ is not isomorphic to $T_{2nm^{2}}^{-1,\, -1}(q)$. In the same manner it can be proved that $T_{2nm^{2}}^{1,\, 1}(q)$ is not isomorphic to $T_{2nm^{2}}^{-1,\, 1}(q)$ and respectively that $T_{2nm^{2}}^{-1,\, -1}(q)$ is not isomorphic to $T_{2nm^{2}}^{-1,\, 1}(q)$. Hence, in this case we have $\sharp_{n,\,m}^{\,\,\,q} = 3$.

Finally, if $m$ is even one can easily conclude by pursuing a similar strategy that all Hopf algebras $T_{2nm^{2}}^{\beta,\, \sigma}(q)$ are in fact isomorphic to $T_{2nm^{2}}^{1,\, 1}(q)$ and therefore $\sharp_{n,\,m}^{\,\,\,q} = 1$.
\end{proof}

\end{document}